\documentclass{article}
\usepackage{amscd}
\usepackage{amsmath}
\usepackage{amssymb}
\usepackage{amsthm}
\usepackage{ascmac}
\usepackage{bm}
\usepackage{graphicx}
\usepackage{here}
\usepackage{hyphenat}
\usepackage{mathtools}\numberwithin{equation}{section}
\usepackage{minitoc}
\usepackage{multicol}
\usepackage{multirow}
\usepackage{setspace}
\theoremstyle{definition}
\newtheorem{df}{Definition}
\newtheorem{note}{Notation}
\newtheorem{q}{Problem}
\newtheorem{eg}{Example}
\newtheorem{rmk}{Remark}
\theoremstyle{theorem}
\newtheorem{prop}{Proposition}
\newtheorem{thm}{Theorem}
\newtheorem{cor}{Corollary}

\DeclareMathOperator{\codeg}{codeg}
\DeclareMathOperator{\ehr}{ehr}

\DeclareMathOperator{\rl}{rl}
\DeclareMathOperator{\sd}{sd}

\DeclareMathOperator{\vol}{vol}

\title{Eventual Polynomiality and Dualities\\
of Relative Ehrhart Functions}
\author{Takashi Hirotsu}
\date{\today}
\begin{document}
\maketitle
\begin{abstract}
Classical Ehrhart theory measures the discrete capacity of a convex rational (or integral) polytope $P$ by counting the number of lattice points in the $t$-th dilate $tP$ of $P$. 
In this paper, we extend this paradigm by replacing a lattice point with a geometric object $Q$ of dimension at most $\dim P$. 
We show that the counting function $\mathrm{ehr}(P,Q;t)$ of such valid translations of $Q$ into $tP$ inherits eventual quasi-polynomiality (or eventual polynomiality) with leading term $\mathrm{vol}(P)t^d$, where $d = \dim P$. 
This result is naturally derived by induction on the dimension, based on the classical quasi-polynomiality (or polynomiality) of Ehrhart functions. 
Similarly, replacing $P$ with its relative interior $P^\circ$ defines $\mathrm{ehr}^\circ(P,Q;t)$, which is also shown to be an eventual quasi-polynomial (or eventual polynomial). 
Furthermore, we prove several formulas for relative Ehrhart functions under the condition that $P = kP_0$ and $Q = lQ_0$ for some integers $k$, $l > 0$, and polytopes $P_0$ and $Q_0$ with $\dim Q_0 > 0$ such that $Q_0$ is inscribed in $P_0$. 
In particular, we prove that under this condition, $\mathrm{ehr}(P,Q;-t) = (-1)^{d}\mathrm{ehr}^\circ (P,Q;t+\rho )$ $(t \gg 0)$ holds for some integer $\rho > 0$ if and only if $k \mid 2l$ via the classical Ehrhart--Macdonald reciprocity law. 
We also prove that under the same condition, $\mathrm{ehr}(P,Q;t) = \mathrm{ehr}^\circ (P,Q;t+\sigma )$ $(t \gg 0)$ holds for some integer $\sigma > 0$ if and only if $\mathrm{ehr}^\circ (P_0;\mathrm{codeg}\,P_0) = 1$ and $k \mid \mathrm{codeg}\,P_0$.
\end{abstract}
\section{Introduction}
\subsection{Motivation}
Before formally defining the relative Ehrhart function, we highlight three main motivations underlying this study.
First, the summatory Ehrhart polynomial of a convex integral polytope $P$ introduced in our previous work measures the number of homothetic copies of $P$ inside $tP$, which can be naturally viewed as the sum of the numbers of lattice translations of $sP$ in $tP$ for $s \in \{ 0,1,\dots,t\}$ (see \cite[Definition~1, Theorem~2]{Hir26b}). 
It is therefore a natural next step to consider the problem of counting the valid lattice translations of a convex integral polytope $Q$ of dimension at most $\dim P$ inside $tP$. 
Second, this formulation serves as a higher-dimensional generalization of classical Ehrhart theory, where a lattice point is replaced by a geometric object $Q$. 
We anticipate that analyzing such higher-dimensional configurations will yield novel structural insights into classical Ehrhart theory itself.
Third, from an applied perspective, extending the first Lake Shibireko problem mentioned below---formulated for nested squares---to general collections of multiple polytopes requires a systematic understanding of the asymptotic and algebraic behavior of such counting functions.
\begin{q}[{First Lake Shibireko problem, \cite[Theorem~2]{Hir22}}]\label{q-1st-shibireko}
In a square $P$, draw grid lines to divide each edge into $t$ equal segments, and consider all squares whose edges lie on these grid lines or the original edges. 
Find the limit of the ratio of the arithmetic mean of the areas of those sub-squares to the entire area of $P$ as $t \to \infty$. 
\end{q}
\begin{rmk}
The author conceived Problem~\ref{q-1st-shibireko} while gazing at the lattice pattern of a camping chair by the lakeside of Lake Shibireko in Yamanashi, Japan.
\end{rmk}
\subsection{Relative Ehrhart Functions}
In light of the above discussion, we define the relative Ehrhart function as follows. 
Let $\mathbb Z_{> 0}$ denote the set of positive integers. 
Let $n \geq 0$ and $r \geq 2$ be integers. 
Let $P_1$, $\dots$, $P_r$ be convex rational polytopes in $\mathbb R^n$ with $\dim P_1 \geq \dots \geq \dim P_r$. 
For each set $S \subset \mathbb R^n$, let $S^\circ$ denote the relative interior of $S$. 
\begin{df}[Multi-variable relative Ehrhart function]
Let $t_1$, $\dots$, $t_r \in \mathbb Z_{> 0}$.
\begin{enumerate}
\item[(1)]
\begin{enumerate}
\item[(i)]
If a sequence of polytopes $C_1 \supset \cdots \supset C_r$ can be chosen in the following manner, we denote the total number of such sequences by $\ehr (P_1,\dots,P_r;t_1,\dots,t_r)$.
\begin{itemize}
\item
Let $C_1 = t_1P_1$.
\item
For $i = 1$, $\dots$, $r-1$ in sequence, choose a polytope $C_{i+1} = t_{i+1}P_{i+1}+\bm a_i$ obtained by translating $t_{i+1}P_{i+1}$ by $\bm a_i \in \mathbb Z^n$ such that $C_i \supset C_{i+1}$.
\end{itemize}
\item[(ii)]
Otherwise, we define $\ehr (P_1,\dots,P_r;t_1,\dots,t_r) = 0$.
\end{enumerate}
\item[(2)]
\begin{enumerate}
\item[(i)]
If a sequence of polytopes $C_1 \supset \cdots \supset C_r$ can be chosen in the following manner, we denote the total number of such sequences by $\ehr^\circ (P_1,\dots,P_r;t_1,\dots,t_r)$.
\begin{itemize}
\item
Let $C_1 = t_1P_1$.
\item
For $i = 1$, $\dots$, $r-1$ in sequence, choose a polytope $C_{i+1} = t_{i+1}P_{i+1}+\bm a_i$ obtained by translating $t_{i+1}P_{i+1}$ by $\bm a_i \in \mathbb Z^n$ such that $C_i{}^\circ \supset C_{i+1}$.
\end{itemize}
\item[(ii)]
Otherwise, we define $\ehr^\circ (P_1,\dots,P_r;t_1,\dots,t_r) = 0$.
\end{enumerate}
\end{enumerate}
As $t_1$, $\dots$, $t_r$ vary, the functions in (1) and (2) are called the {\itshape relative Ehrhart functions}. 
Here, we refer to $C_1 \supset \cdots \supset C_r$ as an {\itshape $r$-layered matryoshka} with $C_i$ as the {\itshape $i$-th container} for each $i \in \{ 1,\dots,r\}$ (see Figure~\ref{fig-1}). 
We also refer to $C_r$ as the {\itshape core}. 
Furthermore, $P_i$, $t_i$, $P_r$ are called the {\itshape type of the $i$-th container}, the {\itshape $i$-th dilation factor}, and the {\itshape type of the core}, respectively.\par
\end{df}
The relative Ehrhart functions are nothing but the counting functions of matryoshkas.
\begin{figure}[h]
\centering
\includegraphics{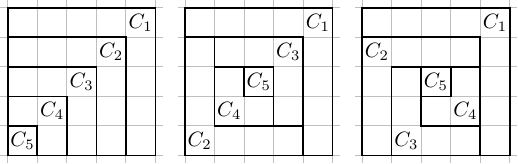}
\caption{Examples of $5$-layered matryoshkas with container types $P_1 = \cdots = P_4 = [0,1]^2$, core type $P_5 = [0,1]^2$, and dilation factors $(t_1,t_2,t_3,t_4,t_5) = (5,4,3,2,1)$.}\label{fig-1}
\end{figure}\par
Hereafter, let $P$ and $Q$ be convex rational polytopes in $\mathbb R^n$ with $\dim P = d \geq \dim Q$. 
\begin{df}[Single-variable relative Ehrhart function]
We define 
\begin{align*} 
\ehr (P,Q;t) &:= \ehr (P,Q;t,1), \\ 
\ehr^\circ (P,Q;t) &:= \ehr^\circ (P,Q;t,1). 
\end{align*} 
\end{df}
\begin{eg}
Let $P$ be a $d$-dimensional convex integral polytope, and let $\bm x$ be a lattice point. 
Then $\ehr (P,\{\bm x\};t)$ and $\ehr^\circ (P,\{\bm x\};t)$ coincide with the classical Ehrhart polynomials $\ehr (P;t) := \# (tP\cap\mathbb Z^n)$ and $\ehr^\circ (P;t) := \# (tP^\circ\cap\mathbb Z^n)$, respectively.
\end{eg}
\begin{eg}\label{eg-cuboid}
Let $P = \prod_{i = 1}^{n}[0,a_i]$ and $Q = \prod_{i = 1}^{n}[0,b_i]$ with $a_i$, $b_i \in \mathbb Z$, $a_i > 0$, and $b_i \geq 0$. 
Then 
\begin{align*} 
\ehr (P,Q;t) &= \prod_{i = 1}^{n}(a_it-b_i+1) \quad \left( t \geq \max_{1 \leq i \leq n}\frac{b_i-1}{a_i}\right), \\ 
\ehr^\circ (P,Q;t) &= \prod_{i = 1}^{n}(a_it-b_i-1) \quad \left( t \geq \max_{1 \leq i \leq n}\frac{b_i+1}{a_i}\right) 
\end{align*} 
(see Proposition~\ref{prop-prism-ehr} for details). 
These equations also hold if $b_1 = \dots = b_n = 0$, namely, if $Q = \{\bm 0\}$.
\end{eg}
\begin{rmk}
Let $d > 0$ and $a \geq 0$ be integers.
\begin{itemize}
\item
If $P = Q = [0,2]^{d-1}\times [0,a+1]$, then $\ehr (P,Q;t) = (2t-1)^{d-1}((a+1)t-a)$.
\item
If $P = [0,1]^{d-1}\times [0,a-1]$, $Q = \{ 0\} ^{d-1}\times [0,a-1]$, and $a > 1$, then $\ehr^\circ (P,Q;t) = (t-1)^{d-1}((a-1)t-a)$ $(t \geq 2)$.
\item
If $P = [0,1]^d$ and $Q = \{\bm 0\}$, then $\ehr^\circ (P,Q;t) = (t-1)^{d}$.
\end{itemize}
These examples imply that the constant terms of $\ehr (P,Q;t)$ and $\ehr^\circ (P,Q;t)$ can take any integer value and any nonzero integer value, respectively.
\end{rmk}
Hereafter, we use the following notation to present further examples.
\begin{note}
\begin{itemize}
\item
Let $\square _d$ denote the $d$-dimensional unit hypercube $[0,1]^d$.
\item
Let $\Delta _d$ denote the $d$-dimensional standard unit simplex, i.e., the simplex with vertices at the origin and the $d$ unit coordinate points.
\item
Let $F_d$ denote the facet of $\Delta _d$ connecting the $d$ unit coordinate points.
\item
Let $L_d$ denote the segment connecting $(0,\dots,0)$ and $(1,\dots,1)$ in $\mathbb R^d$.
\item
On the $xy$-plane, let $O$ denote the origin, let $X_a$ denote the point $(a,0)$, and let $Y_b$ denote the point $(0,b)$.
\item
For each integer $t > 0$, let $S_d(t)$ denote the $t$-th $d$-dimensional simplex
number, which is defined by 
\[ S_0(t) = 1 \quad\text{and}\quad S_d(t) = \sum_{i = 1}^{t}S_{d-1}(i) \quad (d > 0).\]
\end{itemize}
\end{note}
\begin{eg}\label{eg-cb-cb}
Let $d > 0$ be an integer. 
In each of the following cases (see Figure~\ref{fig-2}), we have 
\begin{align*} 
\ehr (P,Q;t) &= t^d, \\ 
\ehr^\circ (P,Q;t) & = (t-2)^d \quad (t \geq 2). 
\end{align*} 
\begin{enumerate}
\item[(1)]
The case where $P = Q = \square _d$.
\item[(2)]
The case where $P = \square _d$ and $Q = \Delta _d$.
\item[(3)]
The case where $P = \square _d$ and $Q = F_d$.
\item[(4)]
The case where $P = \square _d$ and $Q = L_d$. 
\end{enumerate}
\begin{figure}[h]
\centering
\includegraphics{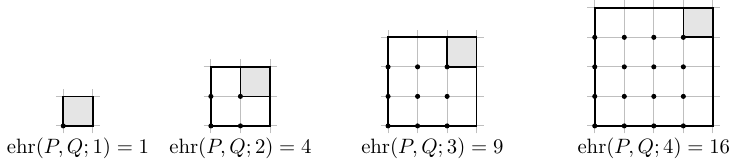}
\caption{The values of $\ehr (P,Q;t)$ for $1 \leq t \leq 4$ in Example~\ref{eg-cb-cb}(1) in the case where $d = 2$.}\label{fig-2}
\end{figure}
\end{eg}
\begin{eg}\label{eg-spx-spx}
Let $d > 0$ be an integer. 
In each of the following cases (see Figure~\ref{fig-3}), we have 
\begin{align*} 
\ehr (P,Q;t) &= S_d(t), \\ 
\ehr^\circ (P,Q;t) &= S_d(t-d-1) \quad (t \geq d). 
\end{align*} 
\begin{enumerate}
\item[(1)]
The case where $P = Q = \Delta _d$.
\item[(2)]
The case where $P = \Delta _d$ and $Q = F_d$.
\end{enumerate}
\begin{figure}[h]
\centering
\includegraphics{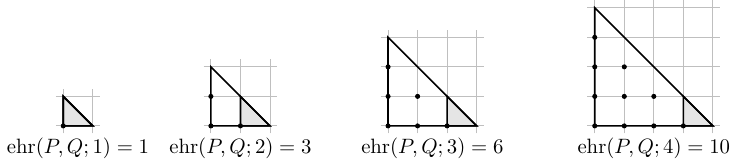}
\caption{The values of $\ehr (P,Q;t)$ for $1 \leq t \leq 4$ in Example~\ref{eg-spx-spx}(1) in the case where $d = 2$.}\label{fig-3}
\end{figure}
\end{eg}
\begin{eg}\label{eg-spx-cb}
Let $d > 0$ be an integer. 
In each of the following cases (see Figure~\ref{fig-4}), we have 
\begin{align*} 
\ehr (P,Q;t) &= S_d(2t-1), \\ 
\ehr^\circ (P,Q;t) & = S_d(2t-2d-1) \quad (t \geq d+1). 
\end{align*} 
\begin{enumerate}
\item[(1)]
The case where $P = 2\Delta _d$ and $Q = \square _d$.
\item[(2)]
The case where $P = 2\Delta _d$ and $Q = L_d$.
\end{enumerate}
\begin{figure}[h]
\centering
\includegraphics{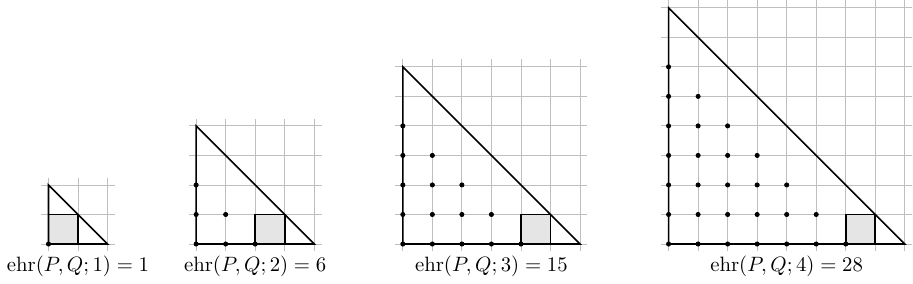}
\caption{The values of $\ehr (P,Q;t)$ for $1 \leq t \leq 4$ in Example~\ref{eg-spx-cb}(1) in the case where $d = 2$.}\label{fig-4}
\end{figure}
\end{eg}
\subsection{Eventual Polynomiality}
The following theorem is the first main result of this paper (see Section~\ref{sec-poly} for the proof). 
Recall that $\dim P = d.$ 
Let $\vol_d(P)$ denote the $d$-dimensional volume of $P$. 
\begin{thm}[Eventual (quasi-)polynomiality]\label{thm-poly-ehr}
\begin{enumerate}
\item[\textup{(1)}]
Both $\ehr (P,Q;t)$ and $\ehr^\circ (P,Q;t)$ are eventual quasi\hyp{}polynomials with leading term $\vol _d(P)t^d$ over $\mathbb Q$.
\item[\textup{(2)}]
If $P$ and $Q$ are integral, then both $\ehr (P,Q;t)$ and $\ehr^\circ (P,Q;t)$ are eventual polynomials over $\mathbb Z[1/d!]$.
\end{enumerate}
\end{thm}
\begin{df}
By Theorem \ref{thm-poly-ehr}, $\ehr (P,Q;t)$ and $\ehr^\circ (P,Q;t)$ are called {\itshape relative Ehrhart eventual quasi-polynomials} if $P$ and $Q$ are rational, and {\itshape relative Ehrhart eventual polynomials} if $P$ and $Q$ are integral.
\end{df}
\begin{rmk}
Let $P = \square _2$ and $Q = a\square _2$ with $a \in \mathbb Z_{> 0}$. 
Then $\ehr (P,Q;t) = 0$ for $0 \leq t < a$ and $\ehr (P,Q;t) = (t-a+1)^2$ for $t \geq a$, so $\ehr (P,Q;t)$ exhibits nonzero polynomial behavior only for $t \geq a-1$. 
Also, in Example~\ref{eg-cuboid}, when $a_1 = a_2 = b_1 = b_2 = 1$, we have $\ehr^\circ (P,Q;1) = 0$ and $\ehr^\circ (P,Q;t) = (t-2)^2$ $(t \geq 2)$. 
In this way, in general, $\ehr (P,Q;t)$ and $\ehr^\circ (P,Q;t)$ become (quasi-)polynomials only for sufficiently large $t$. 
Such functions are called {\itshape eventual (quasi-)polynomials}. 
Hereafter, formulas concerning $\ehr (P,Q;t)$ and $\ehr^\circ (P,Q;t)$ are stated under the assumption that they hold within the range where these functions show nonzero (quasi-)polynomial behavior.
\end{rmk}
\begin{eg}\label{eg-quasi}
Let $P = (3/2)\square _2$ and $Q = \square _2$. 
Then 
\begin{align*} 
\ehr (P,Q;t) &= \begin{cases} 
(3\cdot t/2)^2 & \text{if }t \equiv 0\ (\bmod\ 2), \\ 
(3\cdot (t+1)/2-2)^2 & \text{if }t \equiv 1\ (\bmod\ 2) 
\end{cases} \\ 
&= \begin{cases} 
9t^2/4 & \text{if }t \equiv 0\ (\bmod\ 2), \\ 
(3t-1)^2/4 & \text{if }t \equiv 1\ (\bmod\ 2) 
\end{cases} 
\end{align*} 
becomes a quasi-polynomial with period $2$ (see Figure~\ref{fig-5}).
\begin{figure}[h]
\centering
\includegraphics{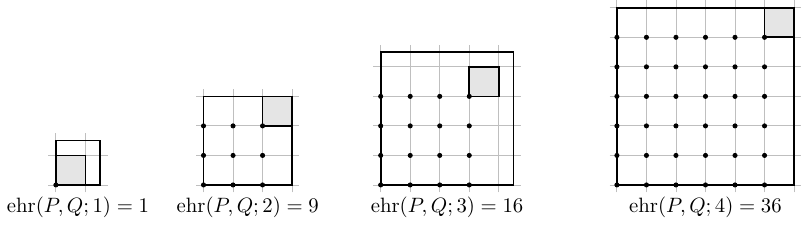}
\caption{The values of $\ehr (P,Q;t)$ for $1 \leq t \leq 4$ in Example~\ref{eg-quasi}.}\label{fig-5}
\end{figure}
\end{eg}
\subsection{Fundamental Functional Equations and Inequalities}
In this paper, we adopt the following condition for $Q$ being inscribed in $P$.
\begin{df}[Inscription of polytopes]
Suppose that $\dim Q > 0$. 
We say that $Q$ is {\itshape inscribed} in $P$ if $P \supset Q$ and every facet of $P$ contains at least one vertex of $Q$.
\end{df}
We consider the following condition, which can be called the {\itshape extended inscription condition}.
\begin{enumerate}
\item[(EI)]
There exist $k$, $l \in \mathbb Z_{>0}$ and polytopes $P_0, Q_0 \subset \mathbb R^n$ such that $P = kP_0$, $Q = lQ_0$, $\dim Q_0 > 0$, and $Q_0$ is inscribed in $P_0$.
\end{enumerate}
The following formulas are fundamental (see Section~\ref{subsec-fe} for the proof). 
\begin{thm}[Fundamental functional equations]\label{thm-kplq-ehr}
Under Condition~\textup{(EI)}, we have 
\begin{align} 
\ehr (P,Q;t) &= \ehr (P_0,Q_0;kt-l+1) = \ehr (P_0;kt-l) \quad \left( t \geq \frac{l}{k}\right), \label{eq-kplq-ehr} \\ 
\ehr^\circ (P,Q;t) &= \ehr^\circ (P_0,Q_0;kt-l+1) = \ehr^\circ (P_0;kt-l) \quad \left( t \gg \frac{l}{k}\right). \label{eq-kplq-ehrc} 
\end{align} 
In particular, if $\dim Q > 0$ and $Q$ is inscribed in $P$, then 
\begin{align} 
\ehr (P,Q;t) &= \ehr (P,P;t) = \ehr (P;t-1), \label{eq-pqi-ehr} \\ 
\ehr^\circ (P,Q;t) &= \ehr^\circ (P,P;t) = \ehr^\circ (P;t-1). \label{eq-pqi-ehrc} 
\end{align} 
\end{thm}
The relative Ehrhart functions are bounded as follows (see Section~\ref{sec-ineq} for the proof). 
\begin{thm}[Absolute inequalities]\label{thm-ineq-ehr}
\begin{enumerate}
\item[\textup{(1)}]
If $k$, $m \in \mathbb Z_{> 0}$ satisfy $kQ+\bm v \subset mP$ for some $\bm v \in \mathbb Z^n$, then 
\begin{align} 
\ehr (Q;kt-1) &\leq \ehr (P,Q;mt), \label{eq-low-ehr} \\ 
\ehr^\circ (Q;kt-1) &\leq \ehr^\circ (P,Q;mt) \label{eq-low-ehrc}
\end{align} 
for $t \gg 0$.
\item[\textup{(2)}]
If $\dim P = \dim Q$ and $l$, $m \in \mathbb Z_{> 0}$ satisfy $mP \subset lQ+\bm w$ for some $\bm w \in \mathbb Z^n$, then 
\begin{align} 
\ehr (P,Q;mt) &\leq \ehr (Q;lt-1), \label{eq-up-ehr} \\ 
\ehr^\circ (P,Q;mt) &\leq \ehr^\circ (Q;lt-1) \label{eq-up-ehrc} 
\end{align} 
for $t \gg 0$.
\end{enumerate}
\end{thm}
\begin{eg}
If $P = Q$, then equality holds in \eqref{eq-low-ehr}--\eqref{eq-up-ehrc} with $k = l = m = 1$ (see \eqref{eq-pp-ehr} and \eqref{eq-pp-ehrc}).
\end{eg}
\begin{eg}
Let $P = \square _2$ and $Q = \Delta _2$. 
Then $\ehr (P,Q;t) = t^2$ as stated in Example~\ref{eg-cb-cb}. 
By $Q \subset P \subset 2Q$, \eqref{eq-low-ehr}, and \eqref{eq-up-ehr}, we have $\ehr (Q;t-1) \leq \ehr (P,Q;t) \leq \ehr (Q;2t-1)$. 
Since 
\[\ehr (Q;t) = \sum_{i = 1}^{t+1}i = \frac{(t+1)(t+2)}{2},\] 
we have 
\[\frac{t(t+1)}{2} \leq \ehr (P,Q;t) \leq t(2t+1) \quad (t \geq 1).\] 
\end{eg}
\subsection{Dualities}
In classical Ehrhart theory, the following theorem plays an important role. 
\begin{thm}[{Ehrhart--Macdonald reciprocity law, \cite[Theorem~4.1]{BR15}}]\label{thm-recip-em}
We have 
\[\ehr (P;-t) = (-1)^{d}\ehr^\circ (P;t).\]
\end{thm}
As a direct consequence of this theorem, we obtain the following theorem.
\begin{thm}[Reciprocity law under the condition $\dim Q = 0$]\label{thm-recip-dim0}
If $\dim Q = 0$, then 
\begin{align*} 
\ehr (P,Q;-t) &= (-1)^{d}\ehr^\circ (P,Q;t), 
\intertext{namely,} 
\ehr^\circ (P,Q;t) &= (-1)^{d}\ehr (P,Q;-t). 
\end{align*} 
\end{thm}
In this paper, we describe several relations between $\ehr (P,Q;t)$ and $\ehr^\circ (P,Q;t)$ in the case where $\dim Q > 0$. 
We use the following definition.
\begin{df}[Reciprocity law]
If there exists an integer $\rho > 0$ such that 
\begin{align} 
\ehr (P,Q;-t) &= (-1)^{d}\ehr^\circ (P,Q;t+\rho ), \label{eq-rl-ec} 
\intertext{namely,} 
\ehr^\circ (P,Q;t) &= (-1)^{d}\ehr (P,Q;-t+\rho ) \label{eq-rl-ce} 
\end{align} 
for $t \gg 0$, then we call Relation~\eqref{eq-rl-ec} the {\itshape reciprocity law} with respect to $P$ and $Q$. 
Furthermore, we call $\rho$ the {\itshape shift number} of the reciprocity law and denote it by $\rl (P,Q)$.
\end{df}
The following theorem is the second main result of this paper (see Section~\ref{subsec-recip} for the proof).
\begin{thm}[Reciprocity law under the extended inscription condition]\label{thm-rli-ehr}
Under Condition~\textup{(EI)}, the reciprocity law \eqref{eq-rl-ec} holds for $t \gg 0$ with $\rl (P,Q) = \rho$ if and only if $k \mid 2l$. 
In this case, $\rl (P,Q) = 2lk^{-1}$.
\end{thm}
\begin{cor}
The shift number $\rl (P,Q) = \rho$ of the reciprocity law \eqref{eq-rl-ec} can take any positive integer value.
\end{cor}
\begin{eg}\label{eg-2cb-cb}
Let $d > 0$ be an integer. 
In each of the following cases (see Figure~\ref{fig-6}), we have 
\begin{align*} 
\ehr (P,Q;t) &= (2t)^d, \\ 
\ehr^\circ (P,Q;t) & = (2t-2)^d \quad (t \geq 1), 
\end{align*} 
and therefore \eqref{eq-rl-ec} holds with $\rl (P,Q) = 1$.
\begin{enumerate}
\item[(1)]
The case where $P = 2\square _d$ and $Q = \square _d$.
\item[(2)]
The case where $P = 2\square _d$ and $Q = \Delta _d$.
\item[(3)]
The case where $P = 2\square _d$ and $Q = F_d$.
\item[(4)]
The case where $P = 2\square _d$ and $Q = L_d$.
\end{enumerate}
\begin{figure}[h]
\centering
\includegraphics{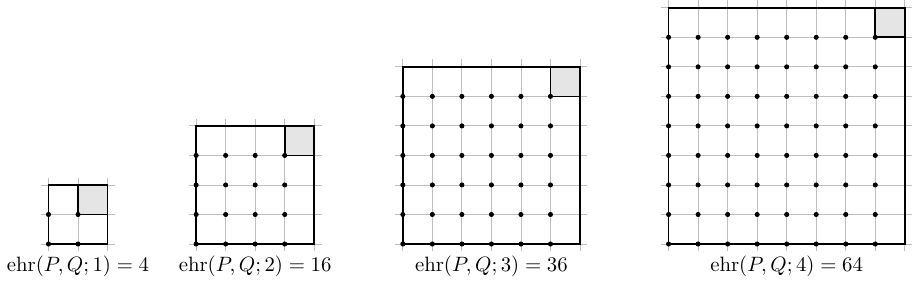}
\caption{The values of $\ehr (P,Q;t)$ for $1 \leq t \leq 4$ in Example~\ref{eg-2cb-cb}(1) in the case where $d = 2$.}\label{fig-6}
\end{figure}
\end{eg}
\begin{eg}
If $\dim Q > 0$ and $Q$ is inscribed in $P$, then \eqref{eq-rl-ec} holds with $\rl (P,Q) = 2$.
\end{eg}
As shown below, in the case where $Q$ is not inscribed in $P$, the reciprocity law \eqref{eq-rl-ec} does not necessarily hold for any integer $\rho > 0$.
\begin{eg}\label{eg-trig-sq}
In each of the following cases (see Figure~\ref{fig-7}), we have 
\begin{align*} 
\ehr (P,Q;t) &= t(3t-2), \\ 
\ehr^\circ (P,Q;t) & = (t-1)(3t-5) \quad (t \geq 1). 
\end{align*} 
There is no integer $\rho > 0$ such that $\ehr(P,Q;-t) = (-1)^{d}\ehr^\circ (P,Q;t+\rho )$, since there is no integer $\rho > 0$ such that $t(3t+2) = (t+\rho -1)(3t+3\rho -5)$. 
Furthermore, we have $\ehr (P,Q;t) = \ehr^\circ (P,Q;t+1)$.
\begin{enumerate}
\item[(1)]
The case where $P = \triangle OX_3Y_2$ and $Q = \square _2$.
\item[(2)]
The case where $P = \triangle OX_3Y_2$ and $Q = L_2$.
\end{enumerate}
\begin{figure}[h]
\centering
\includegraphics{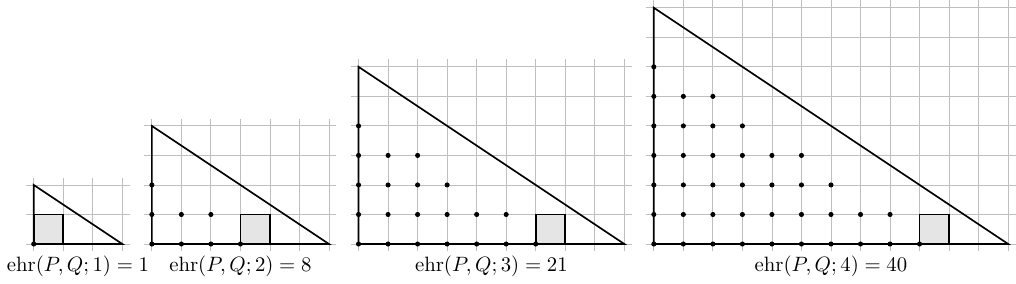}
\caption{The values of $\ehr (P,Q;t)$ for $1 \leq t \leq 4$ in Example~\ref{eg-trig-sq}(1).}\label{fig-7}
\end{figure}
\end{eg}
We also use the following definition.
\begin{df}[Shifted duality]
If there exists an integer $\sigma > 0$ such that 
\begin{align} 
\ehr (P,Q;t) &= \ehr^\circ (P,Q;t+\sigma ), \label{eq-sd-ec} 
\intertext{namely,} 
\ehr^\circ (P,Q;t) &= \ehr (P,Q;t-\sigma ) \label{eq-sd-ce} 
\end{align} 
for $t \gg 0$, then we call Relation~\eqref{eq-sd-ec} the {\itshape shifted duality} with respect to $P$ and $Q$. 
Furthermore, we call $\sigma$ the {\itshape shift number} of the shifted duality and denote it by $\sd (P,Q)$.
\end{df}
The following theorem is the third main result of this paper (see Section~\ref{subsec-sd} for the proof). 
Let $\codeg P$ denote the codegree of $P$, which is defined by 
\[\codeg P := \min\{ t \in \mathbb Z_{> 0} \mid tP^\circ\cap\mathbb Z^n \neq \varnothing\}.\] 
It is known that $\codeg P \le d+1$ (see \cite[Theorem~4.5]{BR15}).
\begin{thm}[Shifted duality under the extended inscription condition]\label{thm-sdi-ehr}
Under Condition~\textup{(EI)}, the shifted duality \eqref{eq-sd-ec} holds for $t \gg 0$ with $\sd (P,Q) = \sigma$ if and only if $\ehr^\circ (P_0;\codeg P_0) = 1$ and $k \mid \codeg P_0$. 
In this case, $\sd (P,Q) = k^{-1}\codeg P_0$.
\end{thm}
\begin{cor}
The shift number $\sd (P,Q) = \sigma$ of the shifted duality \eqref{eq-sd-ec} is at most $d+1$.
\end{cor}
A few examples of pairs of polytopes $P$ and $Q$ regarding the shift numbers of the reciprocity law and the shifted duality are summarized in Table~\ref{tbl-sd}.
\begin{table}[h]
\centering
\begin{tabular}{c|cc|cc|cc} \hline 
\!\!Example\!\! & $P$ & $Q$ & $\ehr (P,Q;t)$ & $\ehr^\circ (P,Q;t)$ & $\rl (P,Q)$ & $\sd (P,Q)$ \\ \hline 
\ref{eg-cb-cb} & $\square _d$ & $\square _d$, $\Delta _d$, $F_d$, or $L_d$ & $t^d$ & $(t-2)^d$ & $2$ & $2$ \\ \hline 
\ref{eg-2cb-cb} & $2\square _d$ & $\square _d$, $\Delta _d$, $F_d$, or $L_d$ & $(2t)^d$ & $(2t-2)^d$ & $1$ & $1$ \\ \hline 
{} & $2\square _d$ & $2\square _d$, $2\Delta _d$, $2F_d$, or $2L_d$ & $(2t-1)^d$ & $(2t-3)^d$ & $2$ & $1$ \\ \hline 
{} & $3\square _d$ & $\square _d$, $\Delta _d$, $F_d$, or $L_d$ & $(3t)^d$ & $(3t-2)^d$ & None & None \\ \hline 
{} & $3\square _d$ & $2\square _d$, $2\Delta _d$, $2F_d$, or $2L_d$ & $(3t-1)^d$ & $(3t-3)^d$ & None & None \\ \hline 
{} & $3\square _d$ & $3\square _d$, $3\Delta _d$, $3F_d$, or $3L_d$ & $(3t-2)^d$ & $(3t-4)^d$ & $2$ & None \\ \hline 
\ref{eg-spx-spx} & $\Delta _d$ & $\Delta _d$ or $F_d$ & $S_d(t)$ & $S_d(t-d-1)$ & $2$ & $d+1$ \\ \hline 
{} & $2\Delta _d$ & $\Delta _d$ or $F_d$ & $S_d(2t)$ & $S_d(2t-d-1)$ & None & None \\ \hline 
{} & $2\Delta _d$ & $2\Delta _d$ or $2F_d$ & $S_d(2t-1)$ & $S_d(2t-d-2)$ & $2$ & None \\ \hline 
{} & $(d+1)\Delta _d$ & $(d+1)\Delta _d$ or $(d+1)F_d$ & $S_d((d+1)t-d)$ & \!\!$S_d((d+1)t-2d-1)$\!\! & $2$ & $1$ \\ \hline 
\ref{eg-spx-cb} & $2\Delta _d$ & $\square _d$ or $L_d$ & $S_d(2t-1)$ & $S_d(2t-2d-1)$ & $2$ & $d$ \\ \hline 
\ref{eg-trig-sq} & $\triangle OX_3Y_2$ & $\square _2$ or $L_2$ & $t(3t-2)$ & $(t-1)(3t-5)$ & None & $1$ \\ \hline 
\end{tabular}
\caption{Examples of shift numbers.}\label{tbl-sd}
\end{table}\par
Note that in the case where both the reciprocity law and the shifted duality hold, combining their relations yields a self-duality, i.e., a functional equation for each of $\ehr (P,Q;t)$ and $\ehr^\circ (P,Q;t)$. 
For example, we have 
\begin{align*} 
\ehr (\square _d,\square _d;t) &= (-1)^{d}\ehr (\square _d,\square _d;-t), \\ 
\ehr^\circ (\square _d,\square _d;t) &= (-1)^{d}\ehr^\circ (\square _d,\square _d;-t+4), \\ 
\ehr (2\square _d,\square _d;t) &= (-1)^{d}\ehr (2\square _d,\square _d;-t), \\ 
\ehr^\circ (2\square _d,\square _d;t) &= (-1)^{d}\ehr^\circ (2\square _d,\square _d;-t+2), \\ 
\ehr (\Delta _d,\Delta _d;t) &= (-1)^{d}\ehr (\Delta _d,\Delta _d;-t-d+1), \\ 
\ehr^\circ (\Delta _d,\Delta _d;t) &= (-1)^{d}\ehr^\circ (\Delta _d,\Delta _d;-t+d+3). 
\end{align*} 
\subsection{Product formulas}
By the following formula, the behavior of the multivariable relative Ehrhart function under the extended inscription condition is reduced to that of the single-variable function (see Section~\ref{sec-prod} for the proof).
\begin{thm}[Product formulas]\label{thm-prod-ehr}
Let $k_1$, $\dots$, $k_r \in \mathbb Z_{> 0}$. 
If $\dim P_r > 0$ and $P_{i+1}$ is inscribed in $P_i$ for each $i \in \{ 1,\dots,r-1\}$, then 
\begin{align} 
\ehr (k_1P_1,\dots,k_rP_r;t_1,\dots,t_r) &= \prod_{i = 1}^{r-1}\ehr (P_i,P_{i+1};k_it_i-k_{i+1}t_{i+1}+1) = \prod_{i = 1}^{r-1}\ehr (P_i;k_it_i-k_{i+1}t_{i+1}), \label{eq-prodi-ehr} \\ 
\ehr^\circ (k_1P_1,\dots,k_rP_r;t_1,\dots,t_r) &= \prod_{i = 1}^{r-1}\ehr^\circ (P_i,P_{i+1};k_it_i-k_{i+1}t_{i+1}+1) = \prod_{i = 1}^{r-1}\ehr^\circ (P_i;k_it_i-k_{i+1}t_{i+1}) \label{eq-prodi-ehrc} 
\end{align} 
for $t_1$, $\dots$, $t_r \gg 0$ such that $k_it_i \geq k_{i+1}t_{i+1}$ for each $i \in \{ 1,\dots,r-1\}$.
\end{thm}
\begin{eg}
Let $d \geq 0$ be an integer. 
By Theorem~\ref{thm-prod-ehr}, we have 
\begin{align*} 
\ehr (3\square _d,2\square _d,\square _d;t_1,t_2,t_3) &= \ehr (\square _d;3t_1-2t_2)\cdot\ehr (\square _d;2t_2-t_3) = (3t_1-2t_2+1)^{d}(2t_2-t_3+1)^{d}, \\  
\ehr^\circ (3\square _d,2\square _d,\square _d;t_1,t_2,t_3) &= \ehr^\circ (\square _d;3t_1-2t_2)\cdot\ehr^\circ (\square _d;2t_2-t_3) = (3t_1-2t_2-1)^{d}(2t_2-t_3-1)^{d}. 
\end{align*} 
\end{eg}
The reciprocity law and the shifted duality for single-variable relative Ehrhart functions under the inscription condition are generalized as follows.
\begin{thm}[Multi-variable reciprocity law under the inscription condition]\label{thm-rli-mehr}
If $\dim P_r > 0$ and $P_{i+1}$ is inscribed in $P_i$ for each $i \in \{ 1,\dots,r-1\}$, then 
\begin{align} 
\ehr (P_1,\dots,P_r;-t_1,\dots,-t_r) &= (-1)^{\dim P_1+\cdots +\dim P_{r-1}}\ehr^\circ (P_1,\dots,P_r;t_1,\dots,t_r) \label{eq-rl-mec} 
\end{align} 
for $t_1$, $\dots$, $t_r \gg 0$.
\end{thm}
\begin{thm}[Multi-variable shifted duality under the inscription condition]\label{thm-sdi-mehr}
Let $\sigma _i = \codeg P_i$ for each $i \in \{ 1,\dots,r-1\}$. 
If $\dim P_r > 0$, and $P_{i+1}$ is inscribed in $P_i$ and $\ehr^\circ (P_i;\sigma _i) = 1$ for each $i \in \{ 1,\dots,r-1\}$, then 
\begin{align} 
\ehr (P_1,\dots,P_r;t_1,\dots,t_r) &= \prod_{i = 1}^{r-1}\ehr^\circ (P_i;t_i+\sigma _i-t_{i+1}) \label{eq-sd-mec} 
\end{align} 
for $t_1$, $\dots$, $t_r \gg 0$.
\end{thm}
\section{Fundamental Properties}\label{sec-prop}
\subsection{Invariance}\label{subsec-inv}
In this subsection, we prove the following propositions.
\begin{prop}\label{prop-equiv-ehr}
Both $\ehr (P,Q;t)$ and $\ehr^\circ (P,Q;t)$ are invariant under lattice equivalence.
\end{prop}
\begin{prop}\label{prop-par-ehr}
Let $\bm v$, $\bm w \in \mathbb Z^n$. 
Then 
\begin{align} 
\ehr (P+\bm v,Q+\bm w;t) &= \ehr (P,Q;t), \label{eq-pvqw-ehr} \\ 
\ehr^\circ (P+\bm v,Q+\bm w;t) &= \ehr^\circ (P,Q;t). \label{eq-pvqw-ehrc} 
\end{align} 
\end{prop}
We use the representation of $\ehr (P,Q;t)$ and $\ehr^\circ (P,Q;t)$ in terms of the Minkowski differences as 
\begin{align} 
\ehr (P,Q;t) &= \#\{\bm a \in \mathbb Z^n \mid tP \supset Q+\bm a\} = \#\big( (tP-Q)\cap\mathbb Z^n\big), \label{eq-mind-ehr} \\ 
\ehr^\circ (P,Q;t) &= \#\{\bm a \in \mathbb Z^n \mid tP^\circ \supset Q+\bm a\} = \#\big( (tP^\circ -Q)\cap\mathbb Z^n\big). \label{eq-mind-ehrc} 
\end{align} 
\begin{proof}[Proof of Proposition~\ref{prop-equiv-ehr}]
Let $\varphi :\mathbb R^n\to\mathbb R^n$ be a lattice isomorphism. 
Then 
\[\ehr (\varphi (P),\varphi (Q);t) = \ehr (P,Q;t)\] 
holds by \eqref{eq-mind-ehr}, since 
\[ t\cdot\varphi (P)-\varphi (Q) = \varphi (tP)-\varphi (Q) = \varphi (tP-Q)\] 
and the lattice points of $tP-Q$ correspond one-to-one with those of $\varphi (tP-Q)$ under $\varphi$. 
Similarly, 
\[\ehr^\circ (\varphi (P),\varphi (Q);t) = \ehr^\circ (P,Q;t)\] 
holds by \eqref{eq-mind-ehrc}, since 
\[ t\cdot\varphi (P)^\circ -\varphi (Q) = \varphi (tP^\circ )-\varphi (Q) = \varphi (tP^\circ -Q). \qedhere\] 
\end{proof}
\begin{proof}[Proof of Proposition~\ref{prop-par-ehr}]
Equation~\eqref{eq-pvqw-ehr} holds, since the number of lattice points in 
\[ t(P+\bm v)-(Q+\bm w) = (tP-Q)+t\bm v-\bm w\] 
is equal to that in $tP-Q$. 
Equation~\eqref{eq-pvqw-ehrc} is proved similarly.
\end{proof}
\subsection{Functional Equations}\label{subsec-fe}
In this subsection, we prove the following propositions and Theorem~\ref{thm-kplq-ehr}.
\begin{prop}\label{prop-pp-ehr}
We have 
\begin{align} 
\ehr (P,P;t) &= \ehr (P;t-1), \label{eq-pp-ehr} \\ 
\ehr^\circ (P,P;t) &= \ehr^\circ (P;t-1). \label{eq-pp-ehrc} 
\end{align} 
\end{prop}
\begin{prop}\label{prop-kpq-ehr}
Let $k \in \mathbb Z_{> 0}$. 
Then 
\begin{align} 
\ehr (kP,Q;t) &= \ehr (P,Q;kt), \label{eq-kpq-ehr} \\ 
\ehr^\circ (kP,Q;t) &= \ehr^\circ (P,Q;kt). \label{eq-kpq-ehrc} 
\end{align} 
Also, if $\dim Q > 0$, then 
\begin{align} 
\ehr (kP,kQ;t) &= \ehr (P,Q;kt-k+1), \label{eq-kpkq-ehr} \\ 
\ehr^\circ (kP,kQ;t) &= \ehr^\circ (P,Q;kt-k+1). \label{eq-kpkq-ehrc} 
\end{align} 
\end{prop}
To prove them, we use the representation of $\ehr (P,Q;t)$ and $\ehr^\circ (P,Q;t)$ in terms of the Minkowski sums as 
\begin{align} 
\ehr (P,Q;t) &= \#\big( ((t-1)P+(P-Q))\cap\mathbb Z^n\big), \label{eq-mins-ehr} \\ 
\ehr^\circ (P,Q;t) &= \#\big( ((t-1)P^\circ +(P-Q))\cap\mathbb Z^n\big). \label{eq-mins-ehrc} 
\end{align} 
\begin{proof}[Proof of Proposition~\ref{prop-pp-ehr}]
Equation~\eqref{eq-pp-ehr} is proved as 
\[\ehr (P,P;t) = \#\big( (t-1)P\cap\mathbb Z^n\big) = \ehr (P;t-1)\] 
by \eqref{eq-mins-ehr}. 
Similarly, Equation~\eqref{eq-pp-ehrc} is proved as 
\[\ehr^\circ (P,P;t) = \#\big( (t-1)P^\circ\cap\mathbb Z^n\big) = \ehr^\circ (P;t-1)\] 
by \eqref{eq-mins-ehrc}.
\end{proof}
\begin{proof}[Proof of Proposition~\ref{prop-kpq-ehr}]
Equation~\eqref{eq-kpq-ehr} follows from 
\[ t(kP) = (kt)P.\] 
Also, Equation~\eqref{eq-kpkq-ehr} follows from 
\[ t(kP)-kQ = (kt)P-kQ = (kt-k)P+k(P-Q) = (kt-k+1)P-Q.\] 
Equations \eqref{eq-kpq-ehrc} and \eqref{eq-kpkq-ehrc} are proved similarly.
\end{proof}
\begin{proof}[Proof of Theorem~\ref{thm-kplq-ehr}]
Suppose that (EI) holds. 
For each set $S \subset \mathbb R^n$, let $S^*$ denote either $S$ or $S^\circ$. 
It suffices to show that the equality 
\[ ktP_0{}^*-lQ_0 = ktP_0{}^*-lP_0\] 
of Minkowski differences holds. 
Let $\bm u \in \mathbb R^n$. 
If $ktP_0{}^* \supset lP_0+\bm u$, then $ktP_0{}^* \supset lQ_0+\bm u$, since $P_0 \supset Q_0$. 
This implies $ktP_0{}^*-lQ_0 \supset ktP_0{}^*-lP_0$.\par
Conversely, suppose that $ktP_0{}^* \supset lQ_0+\bm u$. 
Choose an arbitrary vertex $\bm v$ of $P_0$, a facet $F$ of $P_0$ containing $\bm v$, and a vertex $\bm w$ of $Q_0$ lying on $F$. 
Let $\langle\bm a,\bm x\rangle = b$ be the defining equation of $F$, where $\bm a \in \mathbb Q^n$, $b \in \mathbb Q$, and $\langle\bm p,\bm q\rangle$ denotes the standard inner product of vectors $\bm p$, $\bm q \in \mathbb R^n$. 
Since $l\bm w+\bm u \in ktP_0{}^*$, we have $\langle\bm a,l\bm w+\bm u\rangle \lesseqgtr ktb$. 
Combined with 
\[\langle\bm a,l\bm w+\bm u\rangle = l\langle\bm a,\bm w\rangle +\langle\bm a,\bm u\rangle = lb+\langle\bm a,\bm u\rangle,\] 
we obtain $\langle\bm a,\bm u\rangle \lesseqgtr (kt-l)b$, which implies 
\[\langle\bm a,l\bm v+\bm u\rangle = l\langle\bm a,\bm v\rangle +\langle\bm a,\bm u\rangle = lb+\langle\bm a,\bm u\rangle \lesseqgtr lb+(kt-l)b = ktb.\] 
Therefore, $l\bm v+\bm u \in ktP_0{}^*$. 
Since $\bm v$ was chosen arbitrarily, we obtain $ktP_0{}^* \supset lP_0+\bm u$. 
Thus, we have $ktP_0{}^*-lQ_0 \subset ktP_0{}^*-lP_0$. 
This completes the proof.
\end{proof}
\section{Eventual Polynomiality}\label{sec-poly}
\subsection{Prismatization of Polytopes}
In the proof of Theorem \ref{thm-poly-ehr}, we use the following proposition.
\begin{prop}\label{prop-prism-ehr}
Let $a$, $b \in \mathbb Z$ with $a > 0$ and $b \geq 0$. 
Then 
\begin{align} 
\ehr (P\times [0,a],Q\times [0,b];t) &= (at-b+1)\ehr (P,Q;t) \quad \left( t \geq \frac{b-1}{a}\right), \label{eq-prism-ehr} \\ 
\ehr^\circ (P\times [0,a],Q\times [0,b];t) &= (at-b-1)\ehr^\circ (P,Q;t) \quad \left( t \geq \frac{b+1}{a}\right). \label{eq-prism-ehrc} 
\end{align} 
\end{prop}
\begin{proof}
These equations are shown as 
\begin{align*} 
\ehr (P\times [0,a],Q\times [0,b];t) &= \#\big( (t(P\times [0,a])-Q\times [0,b])\cap\mathbb Z^{n+1}\big) \\ 
&= \#\big( (tP\times [0,at]-Q\times [0,b])\cap\mathbb Z^{n+1}\big) \\ 
&= (at-b+1)\cdot\#\big( (tP-Q)\cap\mathbb Z^n\big) \\ 
&= (at-b+1)\ehr (P,Q;t), \\ 
\ehr^\circ (P\times [0,a],Q\times [0,b];t) &= \#\big( (t(P\times [0,a])^\circ -Q\times [0,b])\cap\mathbb Z^{n+1}\big) \\ 
&= \#\big( ((tP\times [0,at])^\circ -Q\times [0,b])\cap\mathbb Z^{n+1}\big) \\ 
&= (at-b-1)\cdot\#\big( (tP^\circ -Q)\cap\mathbb Z^n\big) \\ 
&= (at-b-1)\ehr^\circ (P,Q;t). \qedhere 
\end{align*} 
\end{proof}
\subsection{Eventual Polynomiality}
In this section, we prove Theorem~\ref{thm-poly-ehr}.
\begin{proof}[Proof of Theorem~\ref{thm-poly-ehr}]
We prove by mathematical induction on $d$ that the statements (1) and (2), and the following statements (3) and (4) hold in the case where $\dim P = d > \dim Q$ or $\dim P = \dim Q = d-1$.
\begin{enumerate}
\item[(3)]
Let $h \in \mathbb Z_{> 0}$. 
Then both $\ehr (h^{-1}P,Q;t)$ and $\ehr^\circ (h^{-1}P,Q;t)$ are eventual quasi-polynomials. 
Furthermore, if $P$ and $Q$ are integral, then both $\ehr (h^{-1}P,Q;t)$ and $\ehr^\circ (h^{-1}P,Q;t)$ are eventual quasi-polynomials with period $h$.
\item[(4)]
If $P$ and $Q$ are integral, and $\dim P = d$, then both $\ehr (P\times [0,1],Q\times\{ 0\};t)$ and $\ehr^\circ (P\times [0,1],Q\times\{ 0\};t)$ are eventual polynomials over $\mathbb Z[1/d!]$.
\end{enumerate}\par
In the case where $d = 1$, since $\dim Q = 0$, statements (1)--(4) hold by the quasi-polynomiality (or polynomiality) of classical Ehrhart functions and the equations 
\begin{align} 
\ehr (P\times [0,1],Q\times\{ 0\};t) &= (t+1)\ehr (P,Q;t), \label{eq-prism1-ehr} \\ 
\ehr^\circ (P\times [0,1],Q\times\{ 0\};t) &= (t-1)\ehr^\circ (P,Q;t) \label{eq-prism1-ehrc} 
\end{align} 
obtained from Proposition~\ref{prop-prism-ehr}.\par
Let $d_0 \in \mathbb Z_{> 0}$. 
Assume that statements (1)--(4) hold for each $d \in \mathbb Z_{> 0}$ with $d < d_0$ in the case where $\dim P = d > \dim Q$ or $\dim P = \dim Q = d-1$. 
Consider the case where $d = d_0$.
\begin{enumerate}
\item[(i)]
The case where $\dim P = d > \dim Q$. 
Let $\ehr^* (P,Q;t)$ denote either $\ehr (P,Q;t)$ or $\ehr^\circ (P,Q;t)$.
\begin{itemize}
\item
Suppose that for some integer $h_0 > 0$, $P$ is a $d$-dimensional pyramid with apex at the origin and base given by a rational (or integral) polytope $P_0$ lying on the hyperplane $x_n = h_0$, and $Q$ lies on $P_0$. 
Then, since the cross section of $tP$ cut by the hyperplane $x_n = s$ is congruent to $(s/h_0)P_0$ for each $s \in \{ 0,1,\dots,h_0t\}$ (see Figure~\ref{fig-8}), we have 
\[\ehr^* (P,Q;t) = \sum_{s = 0}^{h_0t}\ehr^*\left(\frac{s}{h_0}P_0,Q;1\right) = \sum_{s = 1}^{h_0t}\ehr^* (h_0^{-1}P_0,Q;s)+\begin{cases} 
0 & \text{if }\dim Q > 0, \\ 
1 & \text{if }\dim Q = 0.
\end{cases}\] 
\begin{figure}[h]
\centering
\includegraphics{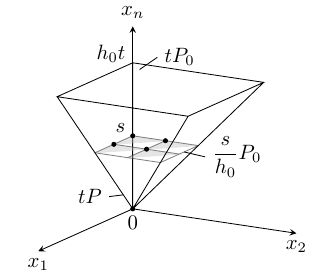}
\caption{The cross section of $tP$ cut by the hyperplane $x_n = s$.}\label{fig-8}
\end{figure}
Here, by the induction hypothesis for (3), $\ehr^* (h_0^{-1}P_0,Q;t)$ is an eventual quasi-polynomial over $\mathbb Q$ (or an eventual quasi-polynomial with period $h_0$ over $\mathbb Z[1/(d-1)!]$). 
This implies that 
\[\ehr^* (P,Q;t) = \sum_{r = 0}^{h_0-1}\sum_{s = 1}^{t}\sum_{i = 0}^{d-1}c_{r,i}^*s^i = \sum_{i = 0}^{d-1}\left(\sum_{r = 0}^{h_0-1}c_{r,i}^*\right)\left(\sum_{s = 1}^{t}s^i\right)\] 
for some $c_{r,i}^* \in \mathbb Q$ (or $c_{r,i}^* \in \mathbb Z[1/(d-1)!]$). 
Therefore, $\ehr^* (P,Q;t)$ is an eventual quasi-polynomial over $\mathbb Q$ (or an eventual polynomial over $\mathbb Z[1/d!]$), since $\sum_{s = 1}^{t}s^i$ is a polynomial in $t$ over $\mathbb Z[1/(i+1)!]$ for each $i \in \{ 0,1,\dots,d-1\}$ by Faulhaber's formula. 
Furthermore, the leading term of $\ehr^* (P,Q;t)$ is 
\[\frac{1}{d}\left(\frac{1}{h_0}\right) ^{d-1}\vol_{d-1}(P_0)(h_0t)^d = \frac{1}{d}\vol_{d-1}(P_0)h_0t^d = \vol_d(P)t^d\] 
by Faulhaber's formula again, since the leading term of $\ehr^* (h_0^{-1}P_0,Q;s)$ is $(h_0^{-1})^{d-1}\vol_{d-1}(P_0)s^{d-1}$. 
To prove (3), let $h \in \mathbb Z_{> 0}$. 
Then, $\ehr^* (h^{-1}h_0^{-1}P_0,Q;t)$ is an eventual quasi-polynomial by the induction hypothesis, which implies that 
\[\ehr^* (h^{-1}P,Q;t) = \sum_{s = 1}^{h_0t}\ehr^* (h^{-1}h_0^{-1}P_0,Q;s)+\begin{cases} 
0 & \text{if }\dim Q > 0, \\ 
1 & \text{if }\dim Q = 0
\end{cases}\] 
is an eventual quasi-polynomial. 
Furthermore, if $P$ and $Q$ are integral, then $P_0$ is also integral, and therefore $\ehr^* (h^{-1}h_0^{-1}P_0,Q;t)$ is an eventual quasi-polynomial with period $h_0h$ by the induction hypothesis, which implies that $\ehr^* (h^{-1}P,Q;t)$ is an eventual quasi-polynomial with period $h$. 
\item
Suppose that for some integers $h_0 > 0$ and $h_1 > 1$, $P$ is a $d$-dimensional pyramid with apex at the origin and base given by a rational polytope $P_0$ lying on the hyperplane $x_n = h_0/h_1$, and $Q$ lies on $P_0$. 
Then, by the above argument, $\ehr^* (h_1P,Q;t)$ is an eventual quasi-polynomial, and so is $\ehr^* (h_1^{-1}h_1P,Q;t) = \ehr^* (P,Q;t)$.
\item
The general case. 
Decompose $P$ into $d$-dimensional pyramids $P_1$, $\dots$, $P_r$ whose bases lie on a hyperplane $H$ parallel to $Q$, such that any two of them intersect in at most a common facet.
For each subset $I$ of $\{ 1,\dots,r\}$, let $P_I = \bigcap_{i \in I}P_i$.
Then, by the principle of inclusion-exclusion, we obtain 
\begin{align*} 
\ehr^* (P,Q;t) &= \sum_{i = 1}^{r}\ehr^* (P_i,Q;t)+f^*(t), \\ 
f^*(t) &= \sum_{\# I = 2}^r(-1)^{\# I-1}\ehr^* (P_I,Q;t).
\end{align*} 
Depending on $i$, there exists a unimodular transformation that maps $H$ onto a hyperplane parallel to a coordinate hyperplane and maps all vertices of $P_i$ to rational points (or lattice points) (see \cite[Sections 1 and 2]{Bar03}). 
By Proposition \ref{prop-equiv-ehr} and the above argument, $\ehr^* (P_i,Q;t)$ is an eventual quasi-polynomial over $\mathbb Q$ (or an eventual polynomial over $\mathbb Z[1/d!]$) with leading term $\vol_d(P_i)t^d$. 
Furthermore, the correction term $f^*(t)$ is an eventual quasi-polynomial over $\mathbb Q$ (or an eventual polynomial over $\mathbb Z[1/(d-1)!]$) of degree at most $d-1$.
Indeed, if $P_I$ has no facet parallel to $H$, then $\ehr^* (P_I,Q;t) = 0$; if $P_I$ has a facet parallel to $H$, then $\ehr^* (P_I,Q;t)$ is an eventual quasi-polynomial over $\mathbb Q$ (or an eventual polynomial over $\mathbb Z[1/(d-1)!]$) of degree at most $d-1$ by the induction hypothesis, since $\dim P_I \leq d-1$.
Therefore, $\ehr^* (P,Q;t)$ is an eventual quasi-polynomial over $\mathbb Q$ (or an eventual polynomial over $\mathbb Z[1/d!]$) with leading term 
\[\sum_{i = 1}^{r}\vol_d(P_i)t^d = \left(\sum_{i = 1}^{r}\vol_d(P_i)\right) t^d = \vol_d(P)t^d,\] 
and satisfies (3).
\end{itemize}
\item[(ii)]
The case where $\dim P = \dim Q = d-1$. 
By (i) and the induction hypothesis for (4), the left-hand sides of \eqref{eq-prism1-ehr} and \eqref{eq-prism1-ehrc} are eventual quasi-polynomials over $\mathbb Q$ (or eventual polynomials over $\mathbb Z[1/(d-1)!]$) with leading term $\vol_d (P\times [0,1])t^d = \vol _{d-1}(P)t^d$, and satisfy (3). 
Therefore, $\ehr (P,Q;t)$ and $\ehr^\circ (P,Q;t)$ are eventual quasi-polynomials over $\mathbb Q$ (or eventual polynomials over $\mathbb Z[1/(d-1)!]$) with leading term $\vol_{d-1}(P)t^{d-1}$, and satisfy (3).
\end{enumerate}
Finally, (4) follows from the induction hypothesis for (2) via the equations \eqref{eq-prism1-ehr} and \eqref{eq-prism1-ehrc}.\par
By the above, the desired assertions hold for all $d \in \mathbb Z_{> 0}$.
\end{proof}
\section{Inequalities}\label{sec-ineq}
\subsection{Monotonicity and Antimonotonicity}
To prove Theorem \ref{thm-ineq-ehr}, we use the following proposition.
\begin{prop}\label{prop-mono-ehr}
Let $P'$ and $Q'$ be convex rational polytopes in $\mathbb R^n$ with $\dim P' \geq \dim Q'$.
\begin{enumerate}
\item[\textup{(1)}]
Monotonicity for the container: if $P \supset P'$, then 
\begin{align} 
\ehr (P,Q;t) &\geq \ehr (P',Q;t), \label{eq-mono-ehr} \\ 
\ehr^\circ (P,Q;t) &\geq \ehr^\circ (P',Q;t). \label{eq-mono-ehrc} 
\end{align} 
\item[\textup{(2)}]
Antimonotonicity for the core: if $Q \supset Q'$, then 
\begin{align} 
\ehr (P,Q;t) &\leq \ehr (P,Q';t), \label{eq-anti-ehr} \\ 
\ehr^\circ (P,Q;t) &\leq \ehr^\circ (P,Q';t). \label{eq-anti-ehrc} 
\end{align} 
\end{enumerate}
\end{prop}
\begin{proof}
Under the assumption $P \supset P'$, Inequality~\eqref{eq-mono-ehr} follows from the implication 
\[ tP' \supset Q+\bm a \Longrightarrow tP \supset Q+\bm a\] 
for each $\bm a \in \mathbb Z^n$. 
Under the assumption $Q \supset Q'$, Inequality~\eqref{eq-anti-ehr} follows from the implication 
\[ tP \supset Q+\bm a \Longrightarrow tP \supset Q'+\bm a\] 
for each $\bm a \in \mathbb Z^n$. 
Inequalities \eqref{eq-mono-ehrc} and \eqref{eq-anti-ehrc} are proved similarly.
\end{proof}
\subsection{Upper and Lower Bounds}
In this subsection, we prove Theorem~\ref{thm-ineq-ehr}.
\begin{proof}[Proof of Theorem~\ref{thm-ineq-ehr}]
Under the assumption of (1), for $t \gg 0$, \eqref{eq-low-ehr} is proved as 
\begin{align*} 
\ehr (P,Q;mt) &= \ehr (mP,Q;t) \quad (\because\eqref{eq-kpq-ehr}) \\ 
&\geq \ehr (kQ,Q;t) \quad (\because kQ \subset mP,\ \text{monotonicity for the container}) \\ 
&= \ehr (Q,Q;kt) \quad (\because\eqref{eq-kpq-ehr}) \\ 
&= \ehr (Q;kt-1) \quad (\because\eqref{eq-pp-ehr}).
\end{align*} 
Furthermore, under the assumption of (2), for $t \gg 0$, \eqref{eq-up-ehr} is proved as 
\begin{align*} 
\ehr (P,Q;mt) &= \ehr (mP,Q;t) \quad (\because\eqref{eq-kpq-ehr}) \\ 
&\leq \ehr (lQ,Q;t) \quad (\because mP \subset lQ,\ \text{monotonicity for the container}) \\ 
&= \ehr (Q,Q;lt) \quad (\because\eqref{eq-kpq-ehr}) \\ 
&= \ehr (Q;lt-1) \quad (\because\eqref{eq-pp-ehr}).
\end{align*} 
The inequalities~\eqref{eq-low-ehrc} and \eqref{eq-up-ehrc} are proved similarly.
\end{proof}
\section{Dualities}
\subsection{Reciprocity Law}\label{subsec-recip}
In this subsection, we prove Theorem~\ref{thm-rli-ehr}. 
\begin{proof}[Proof of Theorem~\ref{thm-rli-ehr}]
Suppose that (EI) holds. 
By Theorems~\ref{thm-kplq-ehr}~and~\ref{thm-recip-em}, we obtain 
\begin{align*} 
\ehr (P,Q;-t) &= \ehr (P_0;-kt-l) \quad (\because\eqref{eq-pqi-ehr}) \\ 
&= (-1)^{d}\ehr^\circ (P_0;kt+l) \quad (\because\text{Theorem}~\ref{thm-recip-em}) 
\intertext{and} 
(-1)^{d}\ehr^\circ (P,Q;t+\rho ) &= (-1)^{d}\ehr^\circ (P_0;k(t+\rho )-l) \quad (\because\eqref{eq-pqi-ehrc}). 
\end{align*} 
Therefore, \eqref{eq-rl-ec} holds with $\rl (P,Q) = \rho$ if and only if $kt+l = k(t+\rho )-l$, or equivalently, $k\rho = 2l$.
\end{proof}
\subsection{Shifted Duality}\label{subsec-sd}
In this subsection, we prove Theorem~\ref{thm-sdi-ehr}. 
\begin{proof}[Proof of Theorem~\ref{thm-sdi-ehr}]
Suppose that (EI) holds. 
Fix an integer $t \gg 0$. 
By Theorem~\ref{thm-kplq-ehr}, \eqref{eq-sd-ec} is equivalent to 
\[\ehr (P_0;kt-l) = \ehr^\circ (P_0;k(t+\sigma )-l).\] 
Furthermore, this is equivalent to 
\[\#\big( (kt-l)P_0\cap\mathbb Z^n\big) = \#\big( (k(t+\sigma )-l)P_0{}^\circ\cap\mathbb Z^n\big),\] 
namely, 
\[\#\big( (kt-l)P_0\cap\mathbb Z^n\big) = \#\big( ((kt-l)P_0+k\sigma P_0{}^\circ )\cap\mathbb Z^n\big).\] 
This equality holds if and only if $\# (k\sigma P_0{}^\circ\cap\mathbb Z^n) = 1$, or equivalently, $\ehr^\circ (P_0;k\sigma ) = 1$. 
Since $\ehr^\circ (P_0;t)$ is strictly increasing for $t \geq \codeg P_0-1$, in this case, we have $k\sigma = \codeg P_0$. 
This completes the proof.
\end{proof}
\section{Product Formulas}\label{sec-prod}
In this section, we prove Theorems~\ref{thm-prod-ehr}--\ref{thm-sdi-mehr}. 
\begin{proof}[Proof of Theorem~\ref{thm-prod-ehr}]
By definition, with $\bm a_0 = \bm 0$, we have 
\begin{align} 
&\ehr (k_1P_1,\dots,k_rP_r;t_1,\dots,t_r) \notag \\ 
&= \#\{ (\bm a_1,\dots,\bm a_{r-1}) \in (\mathbb Z^n)^{r-1} \mid t_i(k_iP_i)+\bm a_{i-1} \supset t_{i+1}(k_{i+1}P_{i+1})+\bm a_i\ (1 \leq i \leq r-1)\} \notag \\ 
&= \#\{ (\bm b_1,\dots,\bm b_{r-1}) \in (\mathbb Z^n)^{r-1} \mid k_it_iP_i \supset k_{i+1}t_{i+1}P_{i+1}+\bm b_i\ (1 \leq i \leq r-1)\} \notag \\ 
&= \prod_{i = 1}^{r-1}\#\{\bm b_i \in \mathbb Z^n \mid k_it_iP_i \supset k_{i+1}t_{i+1}P_{i+1}+\bm b_i\} \notag \\ 
&= \prod_{i = 1}^{r-1}\ehr (k_it_iP_i,k_{i+1}t_{i+1}P_{i+1};1) \notag \\ 
&= \prod_{i = 1}^{r-1}\ehr (P_i,k_{i+1}t_{i+1}P_{i+1};k_it_i) \quad (\because\eqref{eq-kpq-ehr}). \label{eq-prod-ehr} 
\end{align} 
Similarly, we have 
\begin{equation} 
\ehr^\circ (k_1P_1,\dots,k_rP_r;t_1,\dots,t_r) = \prod_{i = 1}^{r-1}\ehr^\circ (P_i,k_{i+1}t_{i+1}P_{i+1};k_it_i) \label{eq-prod-ehrc} 
\end{equation} 
by \eqref{eq-kpq-ehrc}. 
Furthermore, if $\dim P_r > 0$ and $P_{i+1}$ is inscribed in $P_i$ for each $i \in \{ 1,\dots,r-1\}$, then \eqref{eq-prodi-ehr} and \eqref{eq-prodi-ehrc} follow from Theorem~\ref{thm-kplq-ehr}.
\end{proof}
\begin{rmk}
Generally, Equations~\eqref{eq-prod-ehr} and \eqref{eq-prod-ehrc} hold without the condition that $\dim P_r > 0$ and $P_{i+1}$ is inscribed in $P_i$ for each $i \in \{ 1,\dots,r-1\}$.
\end{rmk}
\begin{proof}[Proof of Theorem~\ref{thm-rli-mehr}]
Suppose that $\dim P_r > 0$ and $P_{i+1}$ is inscribed in $P_i$ for each $i \in \{ 1,\dots,r-1\}$. 
Then Theorems~\ref{thm-recip-em} and \ref{thm-prod-ehr} yield 
\begin{align*} 
\ehr (P_1,\dots,P_r;-t_1,\dots,-t_r) &= \prod_{i = 1}^{r-1}\ehr (P_i;-t_i+t_{i+1}) \\ 
&= \prod_{i = 1}^{r-1}(-1)^{\dim P_i}\ehr^\circ (P_i;t_i-t_{i+1}) \\ 
&= (-1)^{\dim P_1+\cdots +\dim P_{r-1}}\prod_{i = 1}^{r-1}\ehr^\circ (P_i;t_i-t_{i+1}) \\ 
&= (-1)^{\dim P_1+\cdots +\dim P_{r-1}}\ehr^\circ (P_1,\dots,P_r;t_1,\dots,t_r) 
\end{align*} 
for $t_1$, $\dots$, $t_r \gg 0$.
\end{proof}
\begin{proof}[Proof of Theorem~\ref{thm-sdi-mehr}]
The desired result follows from Theorems~\ref{thm-sdi-ehr} and \ref{thm-prod-ehr}.
\end{proof}
\section{Open Problems}
The following open problems remain.
\begin{q}
For which polytopes $P$ and $Q$ do $\ehr (P,Q;t)$ and $\ehr^\circ (P,Q;t)$ become polynomials in $t$?
\end{q}
\begin{q}
For which polytopes $P$ and $Q$ does the reciprocity law \eqref{eq-rl-ec} hold for some integer $\rho > 0$?
\end{q}
\begin{q}
For which polytopes $P$ and $Q$ does the shifted duality \eqref{eq-sd-ec} hold for some integer $\sigma > 0$?
\end{q}

\end{document}